\documentclass[12pt,reqno]{amsart}
\usepackage{amscd,amsmath,amsthm,amssymb}
\usepackage{color}
\usepackage{pstricks}
\usepackage{stmaryrd}
\usepackage{tikz}
\usepackage{url}

\usepackage{latexsym}
\usepackage{amsfonts,amsmath,mathtools}
\usepackage{graphics}
\usepackage{float}
\usepackage{enumitem}

\newpsstyle{fatline}{linewidth=1.5pt}
\newpsstyle{fyp}{fillstyle=solid,fillcolor=verylight}
\definecolor{verylight}{gray}{0.97}
\definecolor{light}{gray}{0.9}
\definecolor{medium}{gray}{0.85}
\definecolor{dark}{gray}{0.6}

\def\NZQ{\mathbb}               

\def\ZZ{{\NZQ Z}}
\def\RR{{\NZQ R}}

\def\G{{\mathcal G}}

\def\pd{\textup{proj}\phantom{.}\!\textup{dim}}

\def\opn#1#2{\def#1{\operatorname{#2}}} 
\opn\chara{char} \opn\length{\ell} \opn\pd{pd} \opn\rk{rk}
\opn\projdim{proj\,dim} \opn\injdim{inj\,dim} \opn\rank{rank}
\opn\depth{depth} \opn\grade{grade} \opn\height{height}
\opn\embdim{emb\,dim} \opn\codim{codim}

\opn\Tr{Tr} \opn\bigrank{big\,rank}
\opn\superheight{superheight}\opn\lcm{lcm}
\opn\trdeg{tr\,deg}
\opn\reg{reg} \opn\lreg{lreg} \opn\ini{in} \opn\lpd{lpd}
\opn\size{size} \opn\sdepth{sdepth}
\opn\link{link}\opn\fdepth{fdepth}\opn\lex{lex}
\opn\tr{tr}
\opn\type{type}
\opn\gap{gap}
\opn\diam{diam}
\opn\Mod{Mod}
\opn\div{div} \opn\Div{Div} \opn\cl{cl} \opn\Cl{Cl}
\opn\Spec{Spec} \opn\Supp{Supp} \opn\supp{supp} \opn\Sing{Sing}
\opn\Ass{Ass} \opn\Min{Min}\opn\Mon{Mon}
\opn\Ann{Ann} \opn\Rad{Rad} \opn\Soc{Soc}
\opn\Im{Im} \opn\Ker{Ker} \opn\Coker{Coker} \opn\Am{Am}
\opn\Hom{Hom} \opn\Tor{Tor} \opn\Ext{Ext} \opn\End{End}
\opn\Aut{Aut} \opn\id{id}

\opn\nat{nat}
\opn\pff{pf}
\opn\Pf{Pf} \opn\GL{GL} \opn\SL{SL} \opn\mod{mod} \opn\ord{ord}
\opn\Gin{Gin} \opn\Hilb{Hilb}\opn\sort{sort}
\opn\PF{PF}\opn\Ap{Ap}
\opn\dist{dist}
\opn\aff{aff}
\opn\relint{relint} \opn\st{st}
\opn\lk{lk} \opn\cn{cn} \opn\core{core} \opn\vol{vol}  \opn\inp{inp} \opn\nilpot{nilpot}
\opn\link{link} \opn\star{star}\opn\lex{lex}\opn\set{set}
\opn\width{wd}
\opn\Fr{F}
\opn\QF{QF}
\opn\G{G}
\opn\type{type}\opn\res{res}
\opn\conv{conv}
\opn\sr{sr}
\opn\gr{gr}

\def\pot#1#2{#1[\kern-0.28ex[#2]\kern-0.28ex]}

\opn\dirlim{\underrightarrow{\lim}}
\opn\inivlim{\underleftarrow{\lim}}
\def\Implies{\ifmmode\Longrightarrow \else
	\unskip${}\Longrightarrow{}$\ignorespaces\fi}
\def\implies{\ifmmode\Rightarrow \else
	\unskip${}\Rightarrow{}$\ignorespaces\fi}
\def\iff{\ifmmode\Longleftrightarrow \else
	\unskip${}\Longleftrightarrow{}$\ignorespaces\fi}

\let\:=\colon
\newtheorem{Theorem}{Theorem}[section]
\newtheorem{Lemma}[Theorem]{Lemma}
\newtheorem{Corollary}[Theorem]{Corollary}
\newtheorem{Proposition}[Theorem]{Proposition}

\newtheorem{Example}[Theorem]{Example}

\let\epsilon\varepsilon
\let\kappa=\varkappa
\def\qed{\ifhmode\textqed\fi
	\ifmmode\ifinner\hfill\quad\qedsymbol\else\dispqed\fi\fi}
\def\textqed{\unskip\nobreak\penalty50
	\hskip2em\hbox{}\nobreak\hfill\qedsymbol
	\parfillskip=0pt \finalhyphendemerits=0}
\def\dispqed{\rlap{\qquad\qedsymbol}}

\opn\dis{dis}
\def\pnt{{\raise0.5mm\hbox{\large\bf.}}}

\opn\Lex{Lex}
\opn\Max{Max}
\opn\Shad{Shad}
\opn\astab{astab}

\opn\v{v}

\usepackage{lipsum}

\begin{document}

	\title{The toric ring of one dimensional simplicial complexes}
	\author{Antonino Ficarra, J\"urgen Herzog, Dumitru I. Stamate}
	
	\address{Antonino Ficarra, Department of mathematics and computer sciences, physics and earth sciences, University of Messina, Viale Ferdinando Stagno d'Alcontres 31, 98166 Messina, Italy}
	\email{antficarra@unime.it}
	
	\address{J\"urgen Herzog, Fakult\"at f\"ur Mathematik, Universit\"at Duisburg-Essen, 45117 Essen, Germany} \email{juergen.herzog@uni-essen.de}
	
	\address{Dumitru I. Stamate, Faculty of Mathematics and Computer Science, University of Bucharest, Str. Academiei 14, Bucharest -- 010014, Romania }
	\email{dumitru.stamate@fmi.unibuc.ro}
	
	\thanks{Acknowledgment. This paper was written while the first and the third author visited the Faculty of Mathematics of Essen. D.I. Stamate was partly supported by a grant of the Ministry of Research, Innovation and Digitization, CNCS - UEFISCDI, project number PN-III-P1-1.1-TE-2021-1633, within PNCDI III.
	}
	
	
	
	\maketitle
	
	\begin{abstract}
		Let $\Delta$ be a 1-dimensional simplicial complex. Then $\Delta$ may be identified with a finite simple graph $G$. In this article, we investigate the toric ring $R_G$ of $G$. All graphs $G$ such that $R_G$ is a normal domain are classified. For such a graph, we determine the set $\mathcal{P}_G$ of height one monomial prime ideals of $R_G$. In the bipartite case, and in the case of whiskered cycles, this set is explicitly described. As a consequence, we determine the canonical class $[\omega_{R_G}]$ and characterize the Gorenstein property of $R_G$. For a bipartite graph $G$, we show that $R_G$ is Gorenstein if and only if $G$ is unmixed. For a subclass of non-bipartite graphs $G$, which includes whiskered cycles, $R_G$ is Gorenstein if and only if $G$ is unmixed and has an odd number of vertices. Finally, it is proved that $R_G$ is a pseudo-Gorenstein ring if $G$ is an odd cycle.
	\end{abstract}
	
	\section*{Introduction}
	
	Let $\Delta$ be a simplicial complex on vertex set $[n]=\{1,2,\dots,n\}$. Typically, in Commutative Algebra, one associates to $\Delta$ the Stanley--Reisner ring $S/I_{\Delta}$, where $S=K[x_1,\dots,x_n]$, $K$ is a field and $I_\Delta$ is the Stanley--Reisner ideal of $\Delta$. The theory of Stanley--Reisner ideals has been deeply studied by many researchers. In \cite{HHMQ23a}, the authors introduced a different algebraic object attached to $\Delta$, which they called the toric ring of $\Delta$.
	
	Let $S=K[x_1,\dots,x_n]$ be the polynomial ring with coefficients in a field $K$. For a face $F\in\Delta$, we set ${\bf x}_F=\prod_{i\in F}x_i$ if $F$ is non-empty, otherwise we set ${\bf x}_\emptyset=1$. Then, the \textit{toric ring} of $\Delta$ is defined to be the $K$-subalgebra
	$$
	R_\Delta=K[{\bf x}_Ft:F\in\Delta]
	$$
	of $S[t]$. This algebra is standard graded if we put $\deg(x_1^{a_1}\cdots x_n^{a_n}t^b)=b$, for all monomials $x_1^{a_1}\cdots x_n^{a_n}t^b\in R_\Delta$. This concept was further extended to multicomplexes in \cite{HHMQ23b}, where discrete polymatroids were mainly considered. When $R_\Delta$ is a normal domain, its divisor class group $\textup{Cl}(R_\Delta)$ can be explicitly described in terms of the combinatorics of the pure 1-dimensional skeleton of $\Delta$. This skeleton may be viewed as a graph, which we denote by $G_\Delta$. With such data, one may compute the canonical class $[\omega_{R_\Delta}]$, that is, the class of the canonical module $\omega_{R_\Delta}$ in $\textup{Cl}(R_\Delta)$. Hence, $R_\Delta$ is Gorenstein if and only if $[\omega_{R_\Delta}]=0$. For a Noetherian normal domain $R$, this is one of the most efficient ways to check the Gorenstein property of $R$.
	
	In this article, we consider the toric ring of a 1-dimensional simplicial complex $\Delta$. In this case, the 1-dimensional facets of $\Delta$ are the edges of $G_\Delta$. On the other hand, given a graph $G$ on $[n]$, we may consider the simplicial complex $\Delta$ whose facets are the edges of $G$. Then $G=G_\Delta$. Therefore, we write $R_G$ instead of $R_\Delta$. With this notation, we have $R_G=K[t,x_1t,\dots,x_nt,\{{\bf x}_et\}_{e\in E(G)}]$. We always assume that $G$ has no isolated vertices. To compute the canonical class, one has to determine the set $\mathcal{P}_G$ of height one monomial prime ideals of $R_G$. This is a very difficult task. On the other hand, for the class of bipartite graphs and for certain non-bipartite graphs, including whiskered cycles, we are able to determine such a set. Then, we succeed in classifying the Gorenstein algebras among these classes.
	
	The outline of the article is as follows. In Section \ref{sec1-FHS23}, we summarize the main results proved in \cite{HHMQ23a} about the set $\mathcal{P}_\Delta$ of height one monomial prime ideals of $R_\Delta$. When $R_\Delta$ is normal, then $\omega_{R_\Delta}=\bigcap_{P\in\mathcal{P}_\Delta}P$. Thus, in principle, one can fairly explicitly compute the canonical module and the canonical class. By the facts (iii) and (v) recalled in Page \pageref{ref:iiipage}, $\mathcal{P}_\Delta$ always contains the following set of prime ideals $\mathcal{A}_\Delta=\{P_C:C\in\mathcal{C}(G_\Delta)\}\cup\{Q_1,\dots,Q_n\}$. For the precise definitions of the primes $P_C$ and $Q_i$ see Section \ref{sec1-FHS23}. It is natural to ask when $\mathcal{P}_\Delta=\mathcal{A}_\Delta$. If $R_\Delta$ is normal, this is equivalent to the fact that $\Delta$ is a flag complex and $G_\Delta$ is a perfect graph (Theorem \ref{Thm:PDeltaMin}).
	
	In Section \ref{sec2-FHS23}, we consider the rings $R_G$. In order to apply the machinery developed in Section \ref{sec1-FHS23}, we need to classify the graphs $G$ such that $R_G$ is normal. This is accomplished in Theorem \ref{Thm:RGNormal}. Such a result follows by noting that $R_G$ is isomorphic to the extended Rees algebra of the edge ideal $I(G)$ of $G$, as shown in \cite{DRV10}. Then, by using results in \cite{DRV10,HSV91,OH98,VPV23}, we show that $R_G$ is a normal domain if and only if $G$ has at most one non-bipartite connected component and this component satisfies the so-called odd cycle condition \cite{OH98}. Next, we investigate the set $\mathcal{P}_G$. It turns out that the monomial ideal $P_0=(t,x_1t,\dots,x_nt)$ is always a prime ideal of $R_G$ (Proposition \ref{Prop:extraprime}). For a connected graph $G$, it is proved in Theorem \ref{Thm:Gconnected} that $P_0$ is a non minimal prime ideal of $(t)$ if and only if $G$ is bipartite. These two facts are further equivalent to the property that $\mathcal{P}_G=\mathcal{A}_G$ (Theorem \ref{Thm:Gconnected}(d)). Thus, in the connected bipartite case we know precisely the set $\mathcal{P}_G$. Rephrasing this theorem, we obtain that $P_0$ is a minimal prime if and only if $G$ is non-bipartite (Corollary \ref{Cor:GPGNonBipartite}). 
	
	Hence, one is led to the problem of characterizing the connected non-bipartite graphs $G$ such that $\mathcal{P}_G=\mathcal{A}_G\cup\{P_0\}$. This problem is addressed in Theorem \ref{Thm:PGOddCycle}. For a connected graph $G$, we show that if $\mathcal{P}_G=\mathcal{A}_G\cup\{P_0\}$, then $G$ must be non-bipartite and for any induced odd cycle $G_0$ of $G$, we have that any vertex in $V(G)\setminus V(G_0)$ is adjacent to some vertex of $G_0$. We expect that the converse of this statement holds as well. However, at present we have only partial results supporting this expectation. Therefore, we restrict our attention to unicyclic graphs. In this particular case, we obtain that $\mathcal{P}_G=\mathcal{A}_G\cup\{P_0\}$ if and only if $G$ is a whiskered odd cycle (Theorem \ref{Thm:unicyclic}).
	
	Finally, in the last section we discuss the Gorenstein property of the rings $R_G$. By combining some of the results from \cite{HHMQ23a} a very general criterion for the Gorensteiness of $R_\Delta$ is stated (Theorem \ref{Thm:Georgio}). Then, we apply this result to our rings $R_G$, in the case that $G$ is bipartite or $G$ is an odd (whiskered) cycle. Finally, we prove that $R_{G}$ is pseudo-Gorenstein if $G$ is an odd cycle (Proposition \ref{Prop:Georgio3}).

	\section{Generalities about toric rings of simplicial complexes}\label{sec1-FHS23}
		In the section we summarize some basic facts from \cite{HHMQ23a} about toric rings of simplicial complexes. Let $K$ be a field. Then, the \textit{toric ring of a simplicial complex $\Delta$} on vertex set $[n]$ is defined as the toric ring
	$$
	R_\Delta=K[{\bf x}_Ft:F\in\Delta]\ \subset\ K[x_1,\dots,x_n,t],
	$$
	where we set ${\bf x}_F=\prod_{i\in F}x_i$, if $F$ is nonempty, and ${\bf x}_\emptyset=1$, otherwise.\smallskip
	
	We denote by $G_\Delta$ the graph on vertex set $[n]$ and whose edges are the 1-dimensional faces of $\Delta$. For a graph $G$, we denote by $\mathcal{C}(G)$ the set of the minimal vertex covers of $G$. For a subset $C\subseteq[n]$, we set $\Delta_C=\{F\in\Delta:F\subseteq C\}$.\smallskip
	
	Let $\mathcal{P}_\Delta$ be the set of height one monomial prime ideals of $R_\Delta$. We are interested in this set, because we have $\omega_{R_\Delta}=\bigcap_{P\in\mathcal{P}_\Delta}P$, if $R_\Delta$ is a normal ring, see \cite[Theorem 6.3.5(b)]{BH}. In particular, $[\omega_{R_\Delta}]=\sum_{P\in\mathcal{P}_\Delta}[P]$ in the divisor class group $\textup{Cl}(R_\Delta)$ of $R_\Delta$.\bigskip
	
	Next, we summarize what is known about the set $\mathcal{P}_\Delta$.\medskip
	
	\begin{enumerate}
		\item[(i)] Suppose that $R_\Delta$ is a normal domain. Let $P_1,\dots,P_r$ be the minimal monomial prime ideals of $(t)\subseteq R_\Delta$. Then the classes $[P_1],\dots,[P_r]$ generate the divisor class group $\textup{Cl}(R_\Delta)$ of $R_\Delta$. Furthermore $\textup{Cl}(R_\Delta)$ is free of rank $r-1$ \cite[Theorem 1.1 and Corollary 1.8]{HHMQ23a}.
		\item[(ii)] Let $P$ be a monomial prime ideal of $R_\Delta$, then the set $C=\{i:x_it\in P\}$ is a vertex cover of $G_\Delta$ \cite[Lemma 1.2]{HHMQ23a}.
		\item[(iii)]\label{ref:iiipage} If $C\subseteq[n]$ is a vertex cover of $G_\Delta$, then the ideal $P_C=({\bf x}_Ft:F\in\Delta_C)$ is a prime ideal containing $t$ and it is a minimal prime ideal if and only if $C\in\mathcal{C}(G_\Delta)$ \cite[Theorem 1.3 and Proposition 1.4]{HHMQ23a}.
		\item[(iv)] Not all minimal monomial prime ideals of $(t)$ are of the form $P_C$ for some $C\in\mathcal{C}(G_\Delta)$, see \cite[Example 1.5]{HHMQ23a}.
		\item[(v)] The set of height one monomial prime ideals of $R_\Delta$ not containing $t$ is $\{Q_1,\dots,Q_n\}$, with $Q_i=({\bf x}_Ft:F\in\Delta,i\in F)$ \cite[Proposition 1.9]{HHMQ23a}.
	\end{enumerate}\medskip
	
	By (iii) and (v), the set $\mathcal{P}_\Delta$ of height one monomial prime ideals of $R_\Delta$ contains the set $\{P_C:C\in\mathcal{C}(G_\Delta)\}\cup\{Q_1,\dots,Q_n\}$. In \cite[Theorem 1.10]{HHMQ23a} the authors characterized those simplicial complexes such that this set coincides with $\mathcal{P}_\Delta$ and determined the canonical class $[\omega_{R_\Delta}]$ in such a case \cite[Theorem 1.13]{HHMQ23a}.\medskip
	
	We recall that $\Delta$ is called \textit{flag} if all its minimal nonfaces are of dimension one. Equivalently, $\Delta$ is flag if and only if it is the clique complex of $G_\Delta$.
	
	\begin{Theorem}\label{Thm:PDeltaMin}
		Let $\Delta$ be a simplicial complex on $[n]$. Then, the following conditions are equivalent.
		\begin{enumerate}
			\item[\textup{(a)}] $R_\Delta$ is a normal ring and the set of height one monomial prime ideals of $R_\Delta$ is the set
			$$
			\mathcal{P}_\Delta\ =\ \{P_C:C\in\mathcal{C}(G_\Delta)\}\cup\{Q_1,\dots,Q_n\}.
			$$
			\item[\textup{(b)}] $\Delta$ is a flag complex and $G_\Delta$ is a perfect graph.
		\end{enumerate}
	    
	    Furthermore, if any of these equivalent conditions hold, we have
	    \begin{equation}\label{eq:omegaflagperfect}
	    	[\omega_{R_\Delta}]\ =\ \sum_{C\in\mathcal{C}(G)}(n+1-|C|)[P_C].
	    \end{equation}
	\end{Theorem}

	\section{The bipartite case}\label{sec2-FHS23}
	
	Let $G$ be a graph with no isolated vertices. In this section, we consider the algebras $R_G$.
	
	For a monomial $u=x_1^{a_1}\cdots x_n^{a_n}t^b\in R_G$, we set $\deg_{x_i}(u)=a_i$ for $1\le i\le n$, and $\deg_t(u)=b$. Moreover, if $e=\{i,j\}\in E(G)$, we set ${\bf x}_e=x_ix_j$. 
	
	\begin{Proposition}\label{Prop:extraprime}
		Let $G$ be any graph on $n$ vertices and let $R=R_G$. Then, the ideal $P_0=(t,x_1t,x_2t,\dots,x_nt)$ is a monomial prime ideal of $R$.
	\end{Proposition}
	\begin{proof}
		Since $P_0$ is a monomial ideal, it is enough to prove that for any two monomials $u,v$ not belonging to $P_0$, then the product $uv$ is also not in $P_0$.  Since $u,v\notin P_0$ and $R=K[t,\{x_it\}_{i\in V(G)},\{{\bf x}_et\}_{e\in E(G)}]$, it follows that $uv=\prod_{k=1}^r({\bf x}_{e_k}t)$ for some edges $e_1,\dots,e_r$, not necessarily distinct. Suppose by contradiction that $uv\in P_0$, then $t$ divides $uv$ or $x_jt$ divides $uv$ for some $j$.
		
		In the first case, $uv=tw$ for a suitable monomial $w$. In particular, $\deg_t(w)=r-1$ and $\sum_{i=1}^n\deg_{x_i}(w)=\sum_{i=1}^n\deg_{x_i}(uv)=2r$. Since $\deg_t(w)=r-1$, $w$ is a product of $r-1$ generators of $R$ and we have $\sum_{i=1}^n\deg_{x_i}(w)\le 2(r-1)$, absurd.
		
		Similarly, in the second case we could write $uv=(x_jt)w$ and $\sum_{i=1}^n\deg_{x_i}(w)=\sum_{i=1}^n\deg_{x_i}(uv)-\deg_{x_j}(x_jt)=2r-1$. This is again impossible because $w$ is a product of $r-1$ generators of $R$ and $\sum_{i=1}^n\deg_{x_i}(w)$ is at most $2(r-1)$.
	\end{proof}\medskip
	
	Let $S=K[x_1,\dots,x_n]$ be the standard graded polynomial ring. For a graph $G$, the \textit{edge ideal of $G$} is the ideal $I(G)$ generated by all monomials ${\bf x}_e$ with $e\in E(G)$. Set $I=I(G)$. Recall that the \textit{Rees algebra of $I$} is the $K$-algebra
	$$
	S[It]=\bigoplus_{j\ge0}I^jt^j=K[x_1,\dots,x_n,\{{\bf x}_et\}_{e\in E(G)}]\subset S[t],
	$$
	and the \textit{associated graded ring of $I$} is defined as $\textup{gr}_I(S)=S[It]/IS[It]$.
	
	Whereas, the \textit{extended Rees algebra of $I(G)$} is defined as
	$$
	S[It,t^{-1}]=S[It][t^{-1}]\subset S[t,t^{-1}].
	$$
	
	We have the isomorphism $\varphi:S[It,t^{-1}]\rightarrow R_G$ established by setting $\varphi(t^{-1})=t$, $\varphi(x_i)=x_it$ for $1\le i\le n$, and $\varphi({\bf x}_et)={\bf x}_et$ for $e\in E(G)$, see \cite[Proposition 3.1]{DRV10}.\smallskip
	
	As a first consequence, we  classify all graphs $G$ such that $R_G$ is a normal domain. For this purpose, we recall that a connected graph $G$ is said to satisfy the \textit{odd cycle condition} if for any two induced odd cycles $C_1$ and $C_2$ of $G$, either $C_1$ and $C_2$ have a common vertex or there exist $i\in V(C_1)$ and $j\in V(C_2)$ such that $\{i,j\}\in E(G)$.
	
	\begin{Theorem}\label{Thm:RGNormal}
		Let $G$ be any graph. Then $R_G$ is a normal domain if and only if at most one connected component of $G$ is non-bipartite and this connected component satisfies the odd cycle condition.
	\end{Theorem}
	\begin{proof}
		Let $I=I(G)$. By \cite[Proposition 2.1.2]{HSV91}, the Rees algebra $S[It]$ is normal if and only if the extended Rees algebra $S[It,t^{-1}]$ is normal. Since $R_G\cong S[It,t^{-1}]$, it follows that $R_G$ is normal if and only if $S[It]$ is normal. It is well--known that this is the case, if and only if $I$ is a normal ideal. By \cite[Theorem 8.21]{VPV23}, $I$ is normal if and only if $G$ has at most one non-bipartite connected component $G_i$ and $I(G_i)$ is a normal ideal. By \cite[Theorem 3.3]{DRV10}, $I(G_i)$ is normal if and only $S[I(G_i)t]$ is normal if and only if the toric ring $K[I(G_i)]$ is normal. By \cite[Corollary 2.3]{OH98} this is the case if and only if $G_i$ satisfies the odd cycle condition. The assertion follows.
	\end{proof}\smallskip
	
	Next, we want to algebraically characterize the set of height one monomial prime ideals of $R_G$, for a connected graph $G$. For this aim, note that
	\begin{equation}\label{eq:grvarphi}
		\textup{gr}_I(S)=\frac{S[It]}{IS[It]}\cong\frac{S[It,t^{-1}]}{t^{-1}S[It,t^{-1}]}\cong\frac{R_G}{(t)R_G},
	\end{equation}
	because $t^{-1}$ is mapped to $t$ under the isomorphism $\varphi$.\medskip
	
	\begin{Theorem}\label{Thm:Gconnected}
		Let $G$ be a connected graph with $n$ vertices. Then, the following conditions are equivalent.
		\begin{enumerate}
			\item[\textup{(a)}] The associated graded ring $\textup{gr}_{I(G)}(S)$ is reduced.
			\item[\textup{(b)}] The ideal $(t)\subset R_G$ is radical.
			\item[\textup{(c)}] $G$ is a bipartite graph.
			\item[\textup{(d)}] The set
			$$
			\{P_C:C\in\mathcal{C}(G)\}\cup\{Q_1,\dots,Q_n\}
			$$
			is the set of height one monomial prime ideals of $R_G$.
			\item[\textup{(e)}] The ideal $P_0=(t,x_1t,\dots,x_nt)\subset R_G$ is not a minimal prime of $(t)$.
		\end{enumerate}
	    If any of the above equivalent conditions hold, then $R_G$ is a normal domain.
	\end{Theorem}
	\begin{proof}
		We prove the implications (a)$\iff$(b), (a)$\iff$(c) and (c)$\Rightarrow$(d)$\Rightarrow$(e)$\Rightarrow$(c).
		
		By equation (\ref{eq:grvarphi}) the equivalence (a)$\iff$(b) follows. The equivalence (a)$\iff$(c) is shown in \cite[Proposition 14.3.39]{RV}.
		
		Now, assume (c). Since $G$ is bipartite, it follows that $G$ does not have odd cycles. Thus $R_G$ is a normal domain by Theorem \ref{Thm:RGNormal}. In particular, $G$ is triangle-free. Hence $G$ is a flag complex and a perfect graph, because it is bipartite. Thus, statement (d) follows from Theorem \ref{Thm:PDeltaMin}(b)$\Rightarrow$(a). If (d) holds, then $P_0$ is a monomial prime ideal (Proposition \ref{Prop:extraprime}), but is a not a minimal prime of $(t)$, because $P_0$ is not of the form $P_C$ for any minimal vertex cover $C\in\mathcal{C}(G)$. Statement (e) follows.
		
		Finally, assume (e) and suppose by contradiction that $G$ is non-bipartite. Then $G$ has at least one induced odd cycle $G_1$. By Proposition \ref{Prop:extraprime}, $P_0$ is a monomial prime ideal. By \cite[Corollary 4.33]{BG2009}, the minimal primes ideals containing $(t)$ are monomial prime ideals. Thus, by hypothesis (e), there exists a proper subset $D$ of $V(G)$ such that $Q=(t,\{x_it\}_{i\in D})$ is a minimal prime of $(t)$ and $Q\subsetneq P_0$. It follows that $D$ is a vertex cover of $G$. In particular, $D\cap V(G_1)$ is a vertex cover of $G_1$. Since $G_1$ is an odd cycle, $D$ must contain two adjacent vertices $i,j\in V(G_1)$. Recall that the \textit{distance} of two vertices $p,q\in V(G)$ is defined to be the number $d(p,q)=r$ if there exists a path from $p$ to $q$ of length $r$, that is, a sequence of $r+1$ distinct vertices $p=v_0,v_1,\dots,v_{r-1},v_r=q$ of $G$ such that $\{v_i,v_{i+1}\}\in E(G)$, and no shorter path from $p$ to $q$ exists. If no path between $p$ and $q$ exists, we set $d(p,q)=+\infty$.
		
		Since $G$ is connected and $V(G)\setminus D\ne\emptyset$, the number
		$$
		m\ =\ \min\{d(k,i):k\in V(G)\setminus D\}
		$$
		exists and is finite.
		
		Let $k\in V(G)\setminus D$ such that $d(k,i)=m$. Then, there exists a path of lenght $m$, $i=v_0,v_1,\dots,v_{m-1},v_m=k$. By definition of $m$, it follows that $v_0,v_1,\dots,v_{m-1}\in D$.
		
		If $m\ge2$, then $\{v_{m-2},v_{m-1}\},\{v_{m-1},v_m\}\in E(G)$. Now, $x_{v_{m-2}}x_{v_{m-1}}t,x_{v_m}t\notin Q$, but $(x_{v_{m-2}}x_{v_{m-1}}t)(x_{v_m}t)=(x_{v_{m-2}}t)(x_{v_{m-1}}x_{v_m}t)\in Q$ because $x_{v_{m-2}}t\in Q$. This is a contradiction.
		
		If $m=1$, then $v_1=k$ and $\{i,j\},\{i,k\}\in E(G)$. We have that $x_ix_jt,x_kt\notin Q$. However, $(x_ix_jt)(x_kt)=(x_jt)(x_ix_kt)\in Q$, because $x_jt\in Q$. Again a contradiction. Therefore, $G$ must be bipartite and (c) follows.
		
		Finally, under the equivalent conditions (a)-(e), $G$ is connected and bipartite. The normality of $R_G$ follows from Theorem \ref{Thm:RGNormal}.
	\end{proof}

	An immediate consequence of this result is the following corollary.
	
	\begin{Corollary}\label{Cor:GPGNonBipartite}
		Let $G$ be a connected graph with $n$ vertices. Then $G$ is non-bipartite if and only if
		$$
		(t,x_1t,\dots,x_nt)\in\mathcal{P}_G.
		$$
	\end{Corollary}
	
	\section{The non-bipartite case}\label{sec3-FHS23}
	
	By Corollary \ref{Cor:GPGNonBipartite}, if $G$ is a connected non-bipartite graph on $n$ vertices, we have the inclusion
	\begin{equation}\label{eq:inclusionPG}
		\{P_C:C\in\mathcal{C}(G)\}\cup\{(t,x_1t,\dots,x_nt)\}\cup\{Q_1,\dots,Q_n\}\subseteq\mathcal{P}_G.
	\end{equation}
	
	Thus, it would be interesting to characterize those connected graphs such that equality in (\ref{eq:inclusionPG}) holds. As a first step, we have the following result.
	\begin{Theorem}\label{Thm:PGOddCycle}
		Let $G$ be a connected graph on $n$ vertices such that $R_G$ is a normal domain. Consider the following statements.
		\begin{enumerate}
			\item[\textup{(a)}] The set
			$$
			\mathcal{P}_G=\{P_C:C\in\mathcal{C}(G)\}\cup\{(t,x_1t,x_2t,\dots,x_nt)\}\cup\{Q_1,\dots,Q_n\}
			$$
			is the set of height one monomial prime ideals of $R_G$.
			\item[\textup{(b)}] $G$ is non-bipartite and for any induced odd cycle $G_0$ of $G$, we have that any vertex in $V(G)\setminus V(G_0)$ is adjacent to some vertex of $G_0$.
		\end{enumerate}
	    Then, \textup{(a)} implies \textup{(b)}.
	\end{Theorem}
	
	To prove the theorem, we recall some basic facts about semigroups and semigroup algebras. We denote by $\Delta_G$ the simplicial complex on $[n]$ whose facets are the edges of the graph $G$. As customary, we identify a monomial $x_1^{a_1}\cdots x_n^{a_n}t^b\in R_G$ with its exponent vector $(a_1,\dots,a_n,b)\in\ZZ^{n+1}$. Thus, the monomial $K$-basis of $R_G$ corresponds to the affine semigroup $S\subset\ZZ^{n+1}$ generated by the lattice points $p_F=\sum_{i\in F}e_i+e_{n+1}\in\ZZ^{n+1}$, where $F\in\Delta_G$. Here, $e_1,\dots,e_{n+1}$ is the standard basis of $\ZZ^{n+1}$.
	
	Following \cite{BH}, we denote by $\ZZ S$ the smallest subgroup of $\ZZ^{n+1}$ containing $S$ and by $\RR_+S\subset\RR^{n+1}$ the smallest cone containing $S$. In our case $\ZZ S=\ZZ^{n+1}$. Furthermore, $S=\ZZ^{n+1}\cap\RR_+S$ if $R_G$ is normal \cite[Proposition 6.1.2]{BH}.
	
	A \textit{hyperplane} $H$, defined as the set of solutions of the linear equation $f(x)=a_1x_1+a_2x_2+\dots+a_{n+1}x_{n+1}=0$, is called a \textit{supporting hyperplane} of the cone $\RR_+S$ if $H\cap\RR_+S\ne\emptyset$ and $f({\bf c})\ge0$ for all ${\bf c}\in\RR_+S$. Since any element ${\bf c}\in\RR_+S$ is a linear combination with non-negative coefficients of the lattice points $p_F$, with $F\in\Delta_G$, it follows that $H$ is a supporting hyperplane of $\RR_+S$, if and only if $f(p_F)\ge0$ for all $F\in\Delta_G$.
	
	A subset $\mathcal{F}$ of $\RR_+S$ is called a \textit{face} of $\RR_+S$, if there exists a supporting hyperplane $H$ of $\RR_+S$ such that $\mathcal{F}=H\cap\RR_+S$. We may assume that the coefficients $a_i$ appearing in $f(x)=0$ are integers and $\gcd(a_1,\dots,a_{n+1})=1$. If $H$ is the supporting hyperplane of a facet $\mathcal{F}$, the normalized form defining $H$ is unique and we called it the \textit{support form} of $\mathcal{F}$.
		
	Let $P\subset R_G$ be a monomial ideal. By \cite[Propositions 2.36 and 4.33]{BG2009} we have that $P$ is a monomial prime ideal if and only if there exists a face $\mathcal{F}$ of the cone $\RR_+S$ such that $P=({\bf x}_Ft:F\in\Delta_G\setminus\mathcal{F})$. Equivalently, $P$ is a monomial prime ideal, if and only if there exists a supporting hyperplane $H$ of $\RR_+S$ such that
	$$
	P\ =\ ({\bf x}_Ft:F\in\Delta_G\ \text{and}\ f(p_F)>0).
	$$
	
	\begin{proof}[Proof of Theorem \ref{Thm:PGOddCycle}]
		Assume (a) holds. Then, by Corollary \ref{Cor:GPGNonBipartite}, $G$ is non-bipartite. Hence, $G$ contains at least one induced odd cycle. Suppose for a contradiction that (b) is not satisfied. Then $G$ contains an induced odd cycle $G_0$ and a vertex $v_0\in V(G)\setminus V(G_0)$ that is not adjacent to any vertex $v\in V(G_0)$. After a suitable relabeling, we may assume that $v_0=n$.
		
		We claim that the monomial ideal
		$$
		Q\ =\ (t,x_1t,\dots,x_{n-1}t,\{x_ix_jt\}_{i\in N_G(n),j\in N_G(i)\setminus\{n\}})
		$$
		is a prime ideal of $R_G$. Here for a vertex $k$ of $G$, $N_G(k)$ denotes the set of vertices $i$ such that $\{i,k\}$ is an edge of $G$.
		
		Let $H$ be the hyperplane defined by the equation $f(x)=0$ where
		$$
		f(x)=-\sum_{i\notin N_G(n)}x_i-2x_n+2x_{n+1}.
		$$
		
		Let $F\in\Delta_G$. We claim that $f(p_F)>0$ if ${\bf x}_Ft\in Q$, and $f(p_F)=0$ if ${\bf x}_Ft\notin Q$. This shows that $H$ is a supporting hyperplane of $\RR_+S$ where $S$ is the affine semigroup generated by the lattice points $p_F$, $F\in\Delta_G$, and that $Q=({\bf x}_Ft:F\in\Delta_G,f(p_F)>0)$ is a monomial prime ideal.
		
		If $F=\emptyset$, then $f(p_\emptyset)=2$. Suppose $F=\{i\}$. If $i<n$, then
		$$
		f(p_{\{i\}})=\begin{cases}
			2&\textit{if}\ i\in N_G(n),\\
			1&\textit{if}\ i\notin N_G(n).
		\end{cases}
		$$
		
		If $F=\{n\}$, then $f(p_{\{n\}})=0$.
		
		Finally, assume $F=\{i,j\}\in E(G)$. If $i=n$, then $j\in N_G(n)$ and $f(p_{\{i,j\}})=0$ in this case. Suppose both $i$ and $j$ are different from $n$. Then,
		$$
		f(p_{\{i,j\}})=\begin{cases}
			2&\textit{if}\ \ i,j\in N_G(n),\\
			1&\textit{if}\ \ i\in N_G(n), j\notin N_G(n)\ \ \textit{or}\ \ i\notin N_G(n), j\in N_G(n),\\
			0&\textit{if}\ \ i,j\notin N_G(n).
		\end{cases}
		$$
		
		Therefore, $Q$ is a prime ideal of $R_G$ containing $t$. Thus, there exists a minimal monomial prime ideal $P$ such that $(t)\subset P\subseteq Q$. Hence, $P$ is generated by a subset of the generators of $Q$ and contains $t$. We claim that $P$ is different from $P_C$, for all $C\in\mathcal{C}(G)$, and different from $(t,x_1t,\dots,x_nt)$. This contradicts (a) and shows that (b) holds.
		
		It is clear that $P$ is different from $(t,x_1t,\dots,x_nt)$ because $x_nt\notin P$. Now, let $C\in\mathcal{C}(G)$, then $D=C\cap V(G_0)$ is a vertex cover of $G_0$. Since $G_0$ is an odd cycle, $D$ must contain two adjacent vertices $i,j\in V(G_0)$. Thus, $x_ix_jt\in P_C$. Since $n$ is not adjacent to any vertex $v\in V(G_0)$, we have that $i,j\notin N_G(n)$. Hence $x_ix_jt\notin Q$ and $x_ix_jt\notin P$, also. Thus, $P$ is different from $P_C$, for all $C\in\mathcal{C}(G)$, as wanted.
	\end{proof}
	
	Due to experimental evidence, we expect that statements (a) and (b) of Theorem \ref{Thm:PGOddCycle} are indeed equivalent.
	
	Recall that a graph $G$ is called \textit{unicyclic} if $G$ is connected and contains exactly one induced cycle. Note that a unicyclic graph $G$ satisfies the odd cycle condition, and so $R_G$ is a normal domain. Next, we characterize those unicyclic graphs such that equality holds in (\ref{eq:inclusionPG}). It turns out that for this class of graphs, the statements (a) and (b) of Theorem \ref{Thm:PGOddCycle} are equivalent.
	
	For this aim, we introduce the concept of \textit{whiskered cycles}. Hereafter, for convenience and with abuse of notation, we identify the vertices of $G$ with the variables of $R_G$. Let $k\ge3$ and $a_1,a_2,\dots,a_k\ge0$ be non-negative integers. The \textit{whiskered cycle of type $(a_1,\dots,a_k)$} is the graph $G=C(a_1,\dots,a_k)$ on vertex set
	$$
	V(G)=\{x_1,\dots,x_k\}\cup\bigcup_{i=1}^k\bigcup_{j=1}^{a_i}\{x_{i,j}\},
	$$
	and with edge set
	$$
	E(G)=\{\{x_1,x_2\},\{x_2,x_3\},\dots,\{x_{k-1},x_k\},\{x_k,x_1\}\}\cup\bigcup_{i=1}^k\bigcup_{j=1}^{a_i}\{\{x_i,x_{i,j}\}\}.
	$$
	
	If $k$ is even (odd), $G$ is called a whiskered even (odd) cycle. The vertices $x_{i,j}$ are called the \textit{whiskers} of $x_i$.\smallskip
	
	For example, the whiskered cycle $C(3,2,1,0,1)$ is depicted below\bigskip
	
	
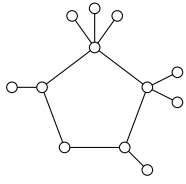
\begin{center}
		\begin{tikzpicture}
			\draw[-] (0.8,0.8) -- (1.1,0) -- (1.9,0) -- (2.2,0.8) -- (1.5,1.33) -- (0.8,0.8);
			\draw[-] (1.5,1.33) -- (1.5,1.85);
			\draw[-] (1.5,1.33) -- (1.2,1.75);
			\draw[-] (1.5,1.33) -- (1.8,1.75);
			\draw[-] (2.2,0.8) -- (2.6,1);
			\draw[-] (2.2,0.8) -- (2.6,0.6);
			\draw[-] (0.8,0.8) -- (0.4,0.8);
			\draw[-] (1.9,0) -- (2.2,-0.3);
			\filldraw[black, fill=white] (0.8,0.8) circle (2pt);
			\filldraw[black, fill=white] (1.1,0) circle (2pt);
			\filldraw[black, fill=white] (1.9,0) circle (2pt);
			\filldraw[black, fill=white] (2.2,0.8) circle (2pt);
			\filldraw[black, fill=white] (1.5,1.33) circle (2pt);
			\filldraw[black, fill=white] (1.5,1.85) circle (2pt);
			\filldraw[black, fill=white] (1.2,1.75) circle (2pt);
			\filldraw[black, fill=white] (1.8,1.75) circle (2pt);
			\filldraw[black, fill=white] (2.6,1) circle (2pt);
			\filldraw[black, fill=white] (2.6,0.6) circle (2pt);
			\filldraw[black, fill=white] (0.4,0.8) circle (2pt);
			\filldraw[black, fill=white] (2.2,-0.3) circle (2pt);
			\filldraw[black, fill=white] (1.5,-0.2) node[below]{};
		\end{tikzpicture}
	\end{center}\medskip

	The next elementary lemma is required.
	\begin{Lemma}\label{Lemma:coversWhiskCycle}
		Let $G=C(a_1,\dots,a_k)$ be a whiskered cycle and $C\in\mathcal{C}(G)$ a minimal vertex cover. If $a_i>0$ for some $i$, then either $x_i\in C$ or $x_{i,j}\in C$ for all $j=1,\dots,a_i$.
	\end{Lemma}
	\begin{proof}
		Let $a_i>0$. Then $x_i$ has at least one whisker. Since $C$ is a minimal vertex cover of $G$, we must have $C\cap\{x_i,x_{i,j}\}\ne\emptyset$ for all $j=1,\dots,a_i$. Suppose $x_i\in C$, then $x_{i,j}\notin C$ for all $j=1,\dots,a_i$, by the minimality of $C$. Otherwise, if $x_i\notin C$, then $x_{i,j}\in C$ for all $j=1,\dots,a_i$, because $C$ is a vertex cover of $G$.
	\end{proof}
	
	Hereafter, we regard the set $[0]$ as the empty set.
	
	Let $G=C(a_1,\dots,a_k)$ be a whiskered cycle. Let $j\ge3$ be a positive integer and let $x_{i},x_{i+1},\dots,x_{i+j}$ be $j+1$ adjacent vertices of the unique induced cycle of $G$. Here, if $i+p$ exceeds $k$, for some $1\le p\le j$, we take the remainder modulo $k$. Then, the \textit{whisker interval} $W(i,i+j)$ is defined as
	\begin{align*}
		W(i,i+j)&=\{x_{i},x_{i+1},\dots,x_{i+j}\}\cup\bigcup_{\ell=i+1}^{i+j-1}\bigcup_{h=1}^{a_\ell}\{x_{\ell,h}\}\\
		&=\{x_{i},x_{i+1},\dots,x_{i+j}\}\cup\{\text{whiskers of}\ x_{i+1},\dots,x_{i+j-1}\}.
	\end{align*}
    We say that $W(i,i+j)$ is \textit{proper} if $\{x_1,x_2,\dots,x_k\}\not\subseteq W(i,i+j)$.
	
	Note that, if $i_1\le i_2\le i_1+j_1-1$ and $i_1+j_1\le i_2+j_2$, then
	$$
	W(i_1,i_1+j_1)\cup W(i_2,i_2+j_2)=W(i_1,i_2+j_2).
	$$
	
	We say that $W(i_1,i_1+j_1)$ and $W(i_2,i_2+j_2)$ are \textit{whisker-disjoint}, if $$|W(i_1,i_1+j_1)\cap W(i_2,i_2+j_2)|\le1,$$ that is $W(i_1,i_1+j_1)$ and $W(i_2,i_2+j_2)$ intersect at most in one vertex.
	
	It is clear that for any collection of proper whisker intervals $W_1,\dots,W_r$ there exist whisker-disjoint proper whisker intervals $V_1,\dots,V_t$ such that $W_1\cup\dots\cup W_r=V_1\cup\dots\cup V_t$.
	
	Now, we are in the position to state and prove the announced classification.
	\begin{Theorem}\label{Thm:unicyclic}
		Let $G$ be a unicyclic graph on $n$ vertices. Then, the following conditions are equivalent.
		\begin{enumerate}
			\item[\textup{(a)}] The set
			$$
			\mathcal{P}_G=\{P_C:C\in\mathcal{C}(G)\}\cup\{(t,x_1t,x_2t,\dots,x_nt)\}\cup\{Q_1,\dots,Q_n\}
			$$
			is the set of height one monomial prime ideals of $R_G$.
			\item[\textup{(b)}] $G$ is a whiskered odd cycle.
		\end{enumerate}
	\end{Theorem}
	\begin{proof}
		Since $G$ is unicyclic, it follows from Theorem \ref{Thm:RGNormal} that $R_G$ is normal.\medskip\\
		The implication (a)$\Rightarrow$(b) follows immediately from Theorem \ref{Thm:PGOddCycle}.\medskip\\
		(b)$\Rightarrow$(a). Suppose $G$ is a whiskered odd cycle. Then $G=C(a_1,\dots,a_k)$ for some odd $k\ge3$ and some non-negative integers $a_1,a_2,\dots,a_k$. Let $G_0$ be the induced graph of $G$ on vertex set $x_1,\dots,x_k$. Then $G_0$ is an odd cycle
		
		Let $P\subset R_G$ be a monomial prime ideal containing $t$ and such that $P\not\supseteq P_C$ for all vertex covers $C$ of $G$. Set $P_0=(t,x_it,x_{i,j}t:i\in[k],j\in[a_i])$. We claim that
		$$
		P_0\subseteq P.
		$$
	
    	The set $D=\{x_i:x_it\in P\}\cup\{x_{i,j}:x_{i,j}t\in P\}$ is a vertex cover of $G$. We are going to prove that $D=V(G)$. From this, it will follow that $P_0\subseteq P$.
    	
    	Since $D$ is a vertex cover, there exists a minimal vertex cover $C$ contained in $D$. By Lemma \ref{Lemma:coversWhiskCycle}, the only adjacent vertices of $C$ can be the vertices of the cycle $G_0$. In particular, $C_0=C\cap V(G_0)$ is a (possibly non minimal) vertex cover of $G_0$.
    	
    	Since $G_0$ is an odd cycle, $C_0$ must contain at least a pair of adjacent vertices $x_i,x_j$ of $G_0$. Suppose that for all such adjacent vertices $x_i,x_j\in C_0$ we have $x_ix_jt\in P$. Then, $P_C$ would be contained in $P$, because by Lemma \ref{Lemma:coversWhiskCycle} the only adjacent vertices of $C$ can be the $x_i$. But this is against our assumption. Therefore, there exist two adjacent vertices $x_i,x_j\in C$ for which $x_ix_jt\notin P$. Up to relabeling, we may assume $i=2$ and $j=3$. We claim that $x_1$ and all the whiskers of $x_2$ and $x_3$ belong to $D$.
    	
    	Suppose that $x_1\notin D$. Then $x_1 t\notin P$. Since also $x_2x_3t\notin P$, the product $(x_1t)(x_2x_3t)$ should not be in $P$. However, $(x_1t)(x_2x_3t)=(x_1x_2t)(x_3t)\in P$, which is a contradiction. Therefore, $x_1\in D$. Similarly, suppose that $x_{2,j}\notin D$ for some $j$. Then $x_{2,j}t\notin P$. Since also $x_2x_3t\notin P$, the product $(x_{2,j}t)(x_2x_3t)$ should not be in $P$. However, $(x_{2,j}t)(x_2x_3t)=(x_2x_{2,j}t)(x_3t)\in P$, a contradiction. Therefore, $x_{2,j}\in D$. Similarly $x_{3,\ell}\in D$ and our claim follows. We distinguish two cases now.\smallskip\\
	    \textsc{Case 1.} Suppose $k=3$. By the previous discussion, $x_1,x_2,x_3,x_{2,j},x_{3,\ell}\in D$, for all $j\in[a_2]$ and $\ell\in[a_3]$. It remains to prove that the whiskers of $x_1$ belong to $D$. Indeed, the vertex cover $C_1=\{x_1,x_2,\text{whiskers of}\ x_3\}$ is contained in $D$. Since $P_{C_1}\not\subseteq P$, we must have $x_1x_2t\notin P$. By the argument used before, we obtain that all whiskers of $x_1$ belong to $D$. Hence, $D=V(G)$ and so $P$ contains $P_0$, as wanted.\smallskip\\
    	\textsc{Case 2.} Suppose $k>3$. By the argument above, we have also that $x_4\in D$. Hence,
    	$$
    	W(1,4)=\{x_1,x_2,x_3,x_4\}\cup\bigcup_{i=2,3}\bigcup_{h\in[a_i]}\{x_{i,h}\}\subseteq D.
    	$$
    	
    	Now, we recursively determine vertex covers $C_i\subset D$ in order to obtain each time new whisker intervals that belong to $D$, and in the end to have that $D=V(G)$.
    	
    	Let
    	$$
    	C_1=\big(C\setminus\{x_2,\text{whiskers of}\  x_1\ \text{and}\ x_4\}\big)\cup\{x_1,x_4\}\cup\{\text{whiskers of}\ x_2\}.
    	$$
    	It is clear that $C_1$ is a cover of $G$. Since, by assumption, $P$ does not contain $P_{C_1}$, it follows that $P$ does not contain $x_ix_jt$, for some adjacent vertices $x_i,x_j\in C_1$. Since $x_2\notin C_1$, it follows that $x_ix_jt$ is different from $x_2x_3t$. Thus, $j=i-1$ and $i\in\{4,\dots,k\}$ or $j=1$ and $i=k$. Let $p$ and $q$ be the adjacent vertices of $i-1$ and $i$, different from $i$ and $i-1$. Then $p=i-2$ and $q=i+1$. Here we take the remainder modulo $k$, if these numbers exceed $k$. Arguing as before,
    	$$
    	W(i-2,i+1)\subseteq D.
    	$$
    	
    	After repeating this argument as many time as possible, if $D=V(G)$ then we are finished. Otherwise, at a given step of this procedure, we have that there exist integers $i_1,j_1$, $\dots$, $i_r,j_r$, with $j_1,\dots,j_r\ge3$ such that
    	\begin{equation}\label{eq:eqWhisk}
    		W(i_1,i_1+j_1)\cup W(i_2,i_2+j_2)\cup\dots\cup W(i_r,i_r+j_r)\subseteq D,
    	\end{equation}
    	and these whisker intervals are proper and mutually whisker-disjoint.
    	
    	Now, starting from the vertex cover $C$, we construct another vertex cover $C'$ of $G$ contained in $D$, having the following properties:
    	\begin{enumerate}
    		\item[(i)] The only adjacent vertices of $C'$ belong to the cycle $G_0$.
    		\item[(ii)] if $x_i,x_j\in C'$ are adjacent vertices that belong to a whisker interval above, say $W(i_a,i_a+j_a)$, then either $\{i,j\}=\{i_a,i_a+1\}$ or $\{i,j\}=\{i_a+j_a-1,i_a+j_a\}$.
    	\end{enumerate}
    
    	The vertex cover $C'$ having the properties (i) and (ii) is constructed as follows. Let $W(i,i+j)$ be a whisker interval in (\ref{eq:eqWhisk}). We distinguish two cases: $j$ even, say $j=2\ell$, and $j$ odd, say $j=2\ell+1$.
    	
    	If $j=2\ell$, we add to $C$ the vertices
    	$$
    	x_i,x_{i+1},x_{i+3},x_{i+5},\dots,x_{i+2\ell-3},x_{i+2\ell-1},x_{i+2\ell}
    	$$
    	and remove all the corresponding whiskers, and moreover, we remove the vertices
    	$$
    	x_{i+2},x_{i+4},\dots,x_{i+2\ell-2}
    	$$
    	and add all the corresponding whiskers. We call $C'$ the resulting set.
    	
    	Whereas, if $j=2\ell+1$, we add to $C$ the vertices
    	$$
    	x_i,x_{i+1},x_{i+3},x_{i+5},\dots,x_{i+2\ell-3},x_{i+2\ell-1},x_{i+2\ell+1}
    	$$
    	and remove all the corresponding whiskers, and moreover, we remove the vertices
    	$$
    	x_{i+2},x_{i+4},\dots,x_{i+2\ell-2},x_{i+2\ell}
    	$$
    	and add all the corresponding whiskers. We call $C'$ the resulting set.
    	
    	When we have more than one whisker interval, we repeat the operations above for all whisker intervals, and call $C'$ the set obtained in this way. Such a set is well defined, because our whisker intervals are proper and mutually whisker-disjoint. It is clear that $C'$ is a vertex cover of $G$ satisfying the properties (i) and (ii).

    	Now, we argue as follows. Since $P$ does not contain $P_{C'}$ by assumption, and since $G_0$ is an odd cycle, by (i) there exists two adjacent vertices $x_i,x_{i+1}\in C'$ such that $x_ix_{i+1}t\notin P$. If $\{i,i+1\}\subseteq W(i_a,i_a+j_a)$, by (ii) either $\{i,i+1\}=\{i_a,i_a+1\}$ or $\{i,i+1\}=\{i_a+j_a-1,i_a+j_a\}$. Say, $\{i,i+1\}=\{i_a,i_a+1\}$, then arguing as before, we have that
    	$$
    	W(i-1,i+2)\subseteq D.
    	$$
    	Otherwise, if $\{i,i+1\}$ is not contained in any of the whisker intervals constructed up to this point, then $W(i-1,i+2)\subseteq D$. In both cases, we can enlarge the set of the whisker intervals contained in $D$. Therefore, after a finite number of steps, we obtain either $D=V(G)$ or a non-proper whisker interval is contained in $D$. In this latter case, up to relabeling we may assume that $W(1,k)\subseteq D$. So, we only need to argue that the whiskers of $x_1$ and $x_k$ are in $D$.
    	
    	Since $W(1,k)\subseteq D$, the vertex cover
    	$$
    	C_2=\{x_1,x_k\}\cup\{x_3,x_5,\dots,x_{k-2}\}\cup\{\text{whiskers of}\ x_{2},x_4,\dots,x_{k-1}\}
    	$$
    	is contained in $D$. Since $P$ does not contain $P_{C_2}$, we must have that $x_1x_kt\notin P$. By the similar argument used before, $W(k-1,2)\subseteq D$. Therefore $D=V(G)$.
    	
    	Since $D=V(G)$, it follows that $P_0\subseteq P$. Therefore, any minimal monomial prime ideal $P$ of $(t)$ different from $P_C$ for all $C\in\mathcal{C}(G)$, must contain $P_0$. Thus $P=P_0$ by Corollary \ref{Cor:GPGNonBipartite}. Hence, the set of height one monomial prime ideals containing $t$ is given by $\{P_C:C\in\mathcal{C}(G)\}\cup\{P_0\}$ and (a) follows.
    \end{proof}
	
	\section{The Gorenstein property}\label{sec4-FHS23}
	
	In this last section, we discuss the Gorenstein property for the toric ring of a simplicial complex $\Delta$. Summarizing some of the results of \cite{HHMQ23a}, we have the following
	\begin{Lemma}\label{Lemma:Georgio}
		Assume that $R_\Delta$ is normal and let $P_1,\dots,P_r$ be the height one monomial prime ideals containing $t$ and $Q_1,\dots,Q_n$ the height one monomial prime ideals not containing $t$. Furthermore, let
		$$
		f_j=\sum_{i=1}^{n+1}c_{i,j}x_i
		$$
		be the support forms associated to $P_j$, $j=1,\dots,r$. Then,
		\begin{enumerate}
			\item[\textup{(a)}] $\textup{Cl}(R_\Delta)$ is generated by $[P_1],\dots,[P_r]$ with unique relation $\sum_{i=1}^rc_{i,n+1}[P_i]=0$.
			\item[\textup{(b)}] For all $j=1,\dots,n$, $[Q_j]=-\sum_{i=1}^rc_{i,j}[P_i]$.
			\item[\textup{(c)}] $[\omega_{R_\Delta}]=\sum_{i=1}^r[P_i]+\sum_{j=1}^n[Q_j]$.
		\end{enumerate}
	\end{Lemma}

	Substituting the expressions for $[Q_j]$ given in (b) into the formula for $[\omega_{R_\Delta}]$ given in (c), we obtain
	$$
	[\omega_{R_\Delta}]\ =\ \sum_{i=1}^r[P_i]-\sum_{j=1}^n\sum_{i=1}^rc_{i,j}[P_i]\ =\ \sum_{i=1}^r\big(1-\sum_{j=1}^nc_{i,j}\big)[P_i].
	$$
	
	Hence, we have proved that
	\begin{Corollary}\label{Cor:Georgio1}
		$
		[\omega_{R_\Delta}]=\sum_{i=1}^{r}(1-\sum_{j=1}^{n}c_{i,j})[P_i].
		$
	\end{Corollary}

	\begin{Theorem}\label{Thm:Georgio}
		The following conditions are equivalent
		\begin{enumerate}
			\item[\textup{(a)}] $R_\Delta$ is Gorenstein.
			\item[\textup{(b)}] There exists an integer $a$ such that
			$
			1-\sum_{j=1}^nc_{i,j}=ac_{i,n+1}
			$
			for all $i=1,\dots,r$.
		\end{enumerate}
	\end{Theorem}
	\begin{proof}
		Observe that $R_\Delta$ is Gorenstein if and only if $[\omega_{R_\Delta}]=0$. By Lemma \ref{Lemma:Georgio}(a) and Corollary \ref{Cor:Georgio1}, this is the case if and only if there exists an integer $a$ such that $1-\sum_{j=1}^nc_{i,j}=ac_{i,n+1}$ for all $i=1,\dots,r$.
	\end{proof}

	Now, we will apply Theorem \ref{Thm:Georgio} to the algebras $R_G$ which we discussed before.

	In the bipartite case, we recover the next result from \cite[Corollary 4.3]{DRV10}.
	
	\begin{Proposition}\label{Prop:Georgio1}
		Let $G$ be a connected bipartite graph on $n$ vertices. Then $R_G$ is Gorenstein if and only if $G$ is unmixed.
	\end{Proposition}
	\begin{proof}
		By Theorem \ref{Thm:Gconnected}(d), $R_G$ is normal and the height one monomial prime ideals containing $t$ are of the form $P_C$, $C\in\mathcal{C}(G)$. In the proof of \cite[Theorem 1.3]{HHMQ23a}, it is shown that the support form associated to $P_C$ is
		\begin{equation}\label{eq:Georgio}
			f_C(x)=-\sum_{i\notin C}x_i+x_{n+1}.
		\end{equation}
		
		Let $\mathcal{C}(G)=\{C_1,\dots,C_r\}$ and $P_i=P_{C_i}$. Then, by Theorem \ref{Thm:Georgio} and formula (\ref{eq:Georgio}) it follows that $R_G$ is Gorenstein if and only if there exists an integer $a$ such that $1-(n-|C_i|)=a$ for all $i=1,\dots,r$. This yields the conclusion.
	\end{proof}
	
	Next, we consider non-bipartite graphs.
	
	\begin{Proposition}\label{Prop:Georgio2}
		Let $G$ be a connected non-bipartite graph with $n$ vertices satisfying the odd cycle condition. Let $\mathcal{C}(G)=\{C_1,\dots,C_r\}$, $P_i=P_{C_i}$, for $i=1,\dots,r$, and $P_0=(t,x_1t,x_2t,\dots,x_nt)$. Assume that the set of height one monomial prime ideals containing $t$ is $\{P_0,P_1,\dots,P_r\}$. Then
		\begin{enumerate}
			\item[\textup{(a)}] $[\omega_{R_G}]=(1-n)[P_0]+\sum_{i=1}^{r}(1-n+|C_i|)[P_i]$.
			\item[\textup{(b)}] $R_G$ is Gorenstein if and only if $n$ is odd and $G$ is unmixed.
		\end{enumerate}
	\end{Proposition}
	\begin{proof}
		One can easily see that the support form of $P_0$ is $f_0(x)=-\sum_{i=1}^nx_i+2x_{n+1}$. Part (a) follows from Corollary \ref{Cor:Georgio1}. By using the support forms $f_0$ and $f_{C_i}$, it follows from Theorem \ref{Thm:Georgio}(b) that $R_G$ is Gorenstein if and only if there exists an integer $a$ such that $1-n=2a$ and $1-n+|C_i|=a$ for all $i$. This implies that $R_G$ is Gorenstein if and only if $n$ is odd and $G$ is unmixed.
	\end{proof}\smallskip

	Finally, we consider the case in which $G$ is a $k$-cycle, which we denote by $C_k$.
	\begin{Corollary}
		$R_{C_k}$ is Gorenstein if and only if $k\in\{3,4,5,7\}$.
	\end{Corollary}
	\begin{proof}
		By Theorem \ref{Thm:RGNormal}, $R_{C_k}$ is normal. We claim that $R_{C_k}$ is Gorenstein if and only $C_k$ is unmixed. If $k$ is even, $C_k$ is bipartite and the claim follows from Proposition \ref{Prop:Georgio1}. If $k$ is odd, the claim follows from Theorem \ref{Thm:unicyclic} and Proposition \ref{Prop:Georgio2}.
		
		It can be easily seen that $C_k$ is unmixed if $k\in\{3,4,5,7\}$. Otherwise, if $k=6$ or $k>7$ then $C_k$ is not unmixed, as we show next.
		
		Let $k>7$ odd, say $k=2\ell+1$. Then, $\ell\ge4$ and
		$$
		\{1,2,4,6,8,10,\dots,2\ell-2,2\ell\},\ \ \ \{1,2,4,5,7,8,10,\dots,2\ell-2,2\ell\}
		$$
		are minimal vertex covers of $C_k$ of size $\ell+1$ and $\ell+2$.
		
		Let $k\ge 6$ even, say $k=2\ell$. If $k=6$, then $\{1,3,5\}$ and $\{1,2,4,5\}$ are minimal vertex covers of $C_6$ of different size. Suppose $\ell\ge4$, then
		$$
		\{1,3,5,7,9,\dots,2\ell-3,2\ell-1\},\ \ \ \{1,2,4,5,7,9,\dots,2\ell-3,2\ell-1\}
		$$
		are minimal vertex covers of $C_k$ of size $\ell$ and $\ell+1$.
	\end{proof}\bigskip

	Let $R$ be a standard graded Cohen--Macaulay $K$-algebra with canonical module $\omega_R$. Following \cite{EHHM15}, we say that $R$ is \textit{pseudo-Gorenstein} if $\dim_K(\omega_R)_a=1$, where $a=\min\{i:(\omega_R)_i\ne0\}$.\smallskip
	
	Let $G$ be a graph such that $R_G$ is a normal domain. By a theorem of Hochster, $R_G$ is a Cohen--Macaulay $K$-algebra. Furthermore, $R_G$ is standard graded with the grading given by $\deg(x_1^{a_1}\cdots x_n^{a_n}t^b)=b$, for all monomials $x_1^{a_1}\cdots x_n^{a_n}t^b\in R_G$.
	
	\begin{Proposition}\label{Prop:Georgio3}
		Let $G$ be an odd cycle. Then $R_G$ is pseudo-Gorenstein.
	\end{Proposition}
    \begin{proof}
    	Let $k$ be the number of vertices of $G$. Then $k=2\ell+1$ for some $\ell\ge1$. Set $P_0=(t,x_1t,\dots,x_kt)$. By Theorem \ref{Thm:unicyclic}, the set of height one monomial prime ideals of $R_G$ is given by $\{P_C:C\in\mathcal{C}(G)\}\cup\{P_0,Q_1,\dots,Q_k\}$, and moreover
    	$$
    	\omega_{R_G}=(\bigcap_{C\in\mathcal{C}(G)}P_C)\cap P_0\cap Q_1\cap\dots\cap Q_k.
    	$$
    	By \cite[Corollary 4.33]{BG2009}, $\omega_{R_G}$ and $\bigcap_{i=1}^kQ_i$ are monomial ideals. Let $u\in\bigcap_{i=1}^kQ_i$ be a monomial. Note that for each $i$, the monomial generators of $Q_i$ have multidegree $\ge e_i+e_{k+1}$. Hence, the multidegree of $u$ is $\ge e_1+\dots+e_k+e_{k+1}$. Thus, $u=u_1u_2\cdots u_b=x_1^{a_1}x_2^{a_2}\cdots x_k^{a_k}t^b$, where $u_1,u_2,\dots,u_b$ are $b$, not necessarily distinct, generators of $R_G$, and $a_1,a_2,\dots,a_k\ge1$. Note that
    	\begin{equation}\label{eq:Georgio26}
    			k\le\sum_{i=1}^k\deg_{x_i}(u)=\sum_{i=1}^k\sum_{j=1}^b\deg_{x_i}(u_j)=\sum_{j=1}^b\sum_{i=1}^k\deg_{x_i}(u_j)\le 2b.
    	\end{equation}
    	Thus $2b\ge k$. Hence $b\ge\ell+1$ and the initial degree of $\bigcap_{i=1}^kQ_i$ is $\ell+1$.
    	
    	We claim that the only monomials of degree $\ell+1$ belonging to $\bigcap_{i=1}^kQ_i$ are
    	\begin{equation}\label{eq:u_0u_i}
    	w_0=(x_1x_2\cdots x_k)t^{\ell+1},\ \ w_i=(x_1\cdots x_{i-1}x_i^2x_{i+1}\cdots x_k)t^{\ell+1},\ \ i=1,\dots,k.
    	\end{equation}
    	
    	Indeed, for all $j=1,\dots,k$, we can write
    	\begin{equation}\label{eq:u_0}
    		w_0=(x_jt)(x_{j+1}x_{j+2}t)\cdots(x_{j+k-2}x_{j+k-1}t)\in Q_j,
    	\end{equation}
        where $j+p$ is understood to be $q$, where $j+p\equiv q$ modulo $k$ and $1\le q\le k$. Thus $w_0\in\bigcap_{i=1}^kQ_i$. Similarly, we can write
        $$
        w_i=(x_{i-1}x_it)(x_ix_{i+1}t)(x_{i+2}x_{i+3}t)\cdots(x_{i+2(\ell-1)}x_{i+2(\ell-1)+1}t)
        $$
        with the same convention as before for the indices. Hence, we see that $w_i\in Q_j$ for all $j$, because $j=i+p$, for some $-1\le p\le 2\ell-1$, and $x_{j-1}x_jt,x_jx_{j+1}t\in Q_j$.
        
        Conversely, let $u=u_1u_2\cdots u_{\ell+1}\in\bigcap_{i=1}^kQ_i$ where $u_1,u_2,\dots,u_{\ell+1}$ are $\ell+1$ generators of $R_G$. Note that at most one of the $u_i$ can be of the form $x_jt$ and the remaining monomials $u_p$ are of the form $x_ix_jt$, otherwise $\sum_{i=1}^k\deg_{x_i}(u)<k$, contradicting (\ref{eq:Georgio26}). Therefore, we have $\sum_{i=1}^k\deg_{x_i}(u)\in\{2\ell+1,2\ell+2\}=\{k,k+1\}$. Since we must have $\deg_{x_i}(u)\ge1$ for all $i=1,\dots,k$, we see that the only monomials of degree $\ell+1$ belonging to $\bigcap_{i=1}^kQ_i$ are those listed in (\ref{eq:u_0u_i}).
        
        Next, we show that $w_0\in P_0\cap(\bigcap_{C\in\mathcal{C}(G)}P_C)$ and $w_i\notin P_0$ for all $i=1,\dots,k$. Indeed, let $C\in\mathcal{C}(G)$, then $x_j\in C$ for some $j$. Thus $x_jt\in P_C$ and by (\ref{eq:u_0}) it follows that $w_0\in P_C$, as well. This same argument shows that $w_0\in P_0$, and so $w_0\in P_0\cap(\bigcap_{C\in\mathcal{C}(G)}P_C)$.
        
        Now let $i\in\{1,\dots,k\}$. For any factorization $w_i=v_1v_2\cdots v_{\ell+1}$ of $w_i$ into a product of generators $v_p\in R_G$, we have $\sum_{j=1}^k\deg_{x_j}(v_p)=2$ for all $p$. This shows that $w_i\notin P_0$.
        
         Therefore, the only monomial of degree $\ell+1$ belonging to $\omega_{R_G}$ is $w_0$. Since $\omega_{R_G}$ is a monomial ideal, its initial degree is larger or equal to the initial degree of $\bigcap_{i=1}^kQ_i$. Hence, $\min\{i:(\omega_{R_G})_i\ne 0\}=\ell+1$ and $\dim_K(\omega_{R_G})_{\ell+1}=1$, that is, $R_G$ is pseudo-Gorenstein.
    \end{proof}\medskip
	
	\begin{Example}
		\rm Let $G=C(1,1,1)$ be the whiskered triangle depicted below.\medskip
		
		
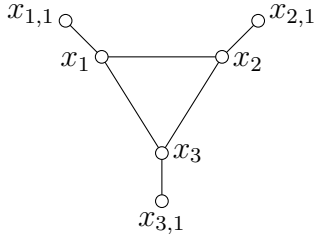
\begin{center}
			\begin{tikzpicture}[scale=0.8]
				\draw[-] (0,0) -- (2,0) -- (1,-1.6) -- (0,0);
				\draw[-] (0,0) -- (-0.6,0.6);
				\draw[-] (2,0) -- (2.6,0.6);
				\draw[-] (1,-1.6) -- (1,-2.4);
				\filldraw[black, fill=white] (0,0) circle (3pt) node[below,left,yshift=-1.8]{$x_1$};
				\filldraw[black, fill=white] (2,0) circle (3pt) node[below,right,yshift=-1.8]{$x_2$};
				\filldraw[black, fill=white] (1,-1.6) circle (3pt) node[below,right]{$x_3$};
				\filldraw[black, fill=white] (-0.6,0.6) circle (3pt) node[above,left,yshift=1.8]{$x_{1,1}$};
				\filldraw[black, fill=white] (2.6,0.6) circle (3pt) node[above,right,yshift=1.8]{$x_{2,1}$};
				\filldraw[black, fill=white] (1,-2.4) circle (3pt) node[below]{$x_{3,1}$};
			\end{tikzpicture}
		\end{center}\smallskip
	
	Note that $G$ is unmixed, but it has an even number of vertices. Thus, by Proposition \ref{Prop:Georgio2} it follows that $R_G$ is not Gorenstein. Indeed, by using \textit{Macaulay2} \cite{GDS}, we checked that the canonical module of $R_G$ is
	$$
	\omega_{R_G}=(x_1x_2x_3x_{1,1}x_{2,1}x_{3,1}t^4,x_1^2x_2^2x_3^2x_{1,1}x_{2,1}x_{3,1}t^5).
	$$
	
	On the other hand, $R_G$ is pseudo-Gorenstein. In general however the algebra $R_G$ of a whisker cycle $G$ need not to be pseudo-Gorenstein. The algebra $R_{C(1,1,2)}$ gives such an example.
	\end{Example}
	
\end{document}